\documentclass[12pt]{article}

\usepackage[utf8]{inputenc}
\usepackage{verbatim}
\usepackage{float}
\usepackage{authblk}
\usepackage{url}
\usepackage{color}
\usepackage{booktabs}
\usepackage[long,c3]{optidef}
\usepackage{fullpage}
\usepackage{amsmath}
\usepackage{amssymb}
\usepackage{amsthm}
\usepackage{cite,hyperref,cleveref}
\usepackage{tikz}
\usetikzlibrary{arrows.meta}
\usepackage[framemethod=tikz]{mdframed}
\usepackage{complexity}
\usepackage{enumerate}
\usepackage[ruled,commentsnumbered,linesnumbered,vlined]{algorithm2e}
\usetikzlibrary{arrows.meta,positioning}
\usepackage{caption}

\newtheorem{lema}{Lemma}

\newtheorem{teo}{Theorem}
\newtheorem{prop}{Proposition}
\newtheorem{cor}{Corollary}
\newtheorem{problem}{Problem}
\theoremstyle{definition}

\DeclareMathOperator{\h}{h}
\DeclareMathOperator{\n}{n}
\DeclareMathOperator{\diam}{diam}
\DeclareMathOperator{\dist}{dist}

\DeclareMathOperator{\ohp}{\ohp\overrightarrow{\h^+}}

\DeclareMathOperator{\vc}{vc}

\usepackage{tabularx}
\usepackage{nameref}
\usepackage[skins,listings]{tcolorbox}

\newtcolorbox{ProblemBox}[2][]{
	colframe=black!20!white,
	coltitle=black,
	arc=1.5mm,
	boxrule=.15mm,
	colbacktitle=white,
	colback=white,
	enhanced,
	adjusted title=flush left,
	attach boxed title to top left={yshift=-3mm,xshift=3mm},
	boxed title style={colframe=white,righttitle=-3mm,lefttitle=-3mm},
	title=#2,#1
}

\begin{document}
\title{Harmonious Colorings: bounds, heuristics and integer-linear formulations\thanks{This work was supported by CNPq 312417/2022-5, 404479/2023-5, 442977/2023-9, 308939/2025-5, 444528/2024-5, 421005/2025-4  and by Projet ANR GODASse, Projet-ANR-24-CE48-4377.}}

\author[1]{J. Araújo}

\author[1]{M. Campêlo}

\affil[1]{\small Universidade Federal do Ceará, Fortaleza, Brazil}

\author[2]{B. Martins\thanks{Corresponding author}}

\affil[2]{\small ENS de Lyon, CNRS, Université Claude Bernard Lyon 1,\newline LIP UMR 5668, 69342 Lyon Cedex 07, France.}

\author[3]{M. C. Santos}

\affil[3]{\small Universidade Federal de Minas Gerais, Belo Horizonte, Brazil}
\affil[1]{\footnotesize{{\texttt{\href{mailto:julio@mat.ufc.br}{julio@mat.ufc.br}, \href{mailto:mcampelo@ufc.br}{mcampelo@ufc.br}, \href{mailto:beatriz.martins@ens-lyon.fr}{beatriz.martins@ens-lyon.fr} and \href{mailto:marciocs@dcc.ufmg.br}{marciocs@dcc.ufmg.br}}}}}

\maketitle

\begin{abstract}
    A proper coloring $c$ of a simple graph $G$ is \emph{harmonious} if, for every pair of distinct edges $uv,xy\in E(G)$, we have that $\{c(u),c(v)\}\neq \{c(x),c(y)\}$. 
    The \emph{harmonious chromatic number} of $G$, denoted by $\h(G)$, is the least positive integer $k$ such that $G$ has a harmonious coloring with $k$ colors.
    In this work, we extend an idea presented in [Kolay, et al. Harmonious coloring: Parameterized algorithms and upper bounds. Theor. Comp. Sci. 772 (2019), 132–142] to compare the harmonious chromatic numbers of two graphs $G$ and $H$, with $H$ being obtained from $G$ by identifying vertices at distance at least three. Furthermore, by fixing a proof presented in the same work, we manage to improve one of its upper bounds.
    We also introduce and study the first, to the best of our knowledge, integer-linear programming formulations for this problem in the literature, along with some heuristics. We provide some preliminary tests on random instances and instances from the second \textit{DIMACS Implementation Challenge}.
\end{abstract}

\section{Introduction}

Given a simple graph $G = (V,E)$, a $k$-coloring of $G$ is a function $c: V(G)\to\{1,\ldots,k\}$. A $k$-coloring $c$ is \emph{proper} if $c(u)\neq c(v)$ for every $uv\in E(G)$. The \emph{chromatic number} of $G$, denoted by $\chi(G)$, is the minimum positive integer $k$ such that $G$ admits a proper $k$-coloring. It is well-known that computing $\chi(G)$ is an $\NP$-hard problem~\cite{Karp72}.

A $k$-coloring $c$ (not necessarily proper) is \emph{line-distinguishing} if, for any pair of distinct edges $uv, xy\in E(G)$, we have that $\{c(u),c(v)\}\neq \{c(x),c(y)\}$. The \emph{line-distinguishing chromatic number}, denoted by $\lambda(G)$, is the least positive integer $k$ for which $G$ admits a line-distinguishing $k$-coloring~\cite{HK1983}. 
A \emph{harmonious $k$-coloring} is a \emph{proper} line-distinguishing $k$-coloring.
 
The \emph{harmonious chromatic number} of $G$, denoted by $\h(G)$, is the minimum positive integer $k$ such that $G$ has a harmonious $k$-coloring~\cite{lee1987upper}.
It is worth mentioning that, in their seminal paper, Hopcroft and Krishnamoorthy~\cite{HK1983} used the term \emph{Harmonious Coloring} to simply mean line-distinguishing coloring; however, since the work of Lee and Mitchem~\cite{lee1987upper}, the term is dedicated to the version of the problem that requires the coloring to be proper.

The distance between two vertices is the length of a shortest path between them (it is $+\infty$ if no such a path exists), and the diameter of a graph is the maximum distance between any two vertices. A harmonious coloring must assign different colors to vertices at distance 1 (since it is proper) or 2 (since it is line-distinguishing). So, complete graphs are not the only $n$-vertex graphs having a harmonious chromatic number equal to $n$. Any graph with $n$ vertices and a diameter of at most two must receive $n$ colors in a harmonious coloring~\cite{MILLER}. 

Graph colorings have been used to model many different types of problems, from computing derivatives~\cite{Gebremedhin05derivate} to testing printed circuits~\cite{Garey76circuit} and performing image segmentation~\cite{Gomez07image}. However, harmonious graph coloring has been studied mainly from a theoretical point of view for several years\cite{importancia}. 
In the scarce literature on direct applications of harmonious coloring, we can cite its use in modeling a radio navigation system \cite{selvi15app1}.

In this work, we present some results on structural graph theory related to harmonious coloring, alongside integer programming models and computational experiments for determining the harmonious chromatic number of a given graph.

\paragraph{Our contributions and related results.}
In the next two sections, we study upper bounds for $\h(G)$. Section 2 considers $d$-degenerate graphs.

We say that a graph $G$ is \emph{d-degenerate} if, for any subgraph $H\subseteq G$, the minimum degree $\delta(H)$ of $H$ is at most $d$.
The Harmonious Coloring of $d$-degenerate graphs has been related to vertex cover~\cite{kolay2019harmonious}. Recall that $S\subseteq V(G)$ is a \emph{vertex cover} of $G$ if, for any edge $uv\in E(G)$, we have $u\in S$ or $v\in S$.
Let $\vc(G)$ denote the cardinality of a minimum vertex cover $S\subseteq V(G)$ of a graph $G$.  
In~\cite{kolay2019harmonious}, the authors claim that if $G$ is $d$-degenerate, then $\h(G)\leq \vc(G)+d(\Delta(G)-1)+\Delta(G)(d-1)$, where $\Delta(G)$ is the maximum degree of $G$. In Section~\ref{sec:2}, we argue that their proof is incorrect and show a counter-example for $d=1$. Besides, we prove that actually $\h(G)\leq \vc(G)+d(\Delta(G)-1)+1$, thus providing a tighter bound for $d\geq 2$.

Section~\ref{sec::theory} considers arbitrary graphs. Using identifications on vertices at distance at least three, we derive a sufficient condition for a graph to have a harmonious chromatic number at most $k$. We remark that part of our argument is similar to the one used in~\cite{kolay2019harmonious} to identify color classes in a harmonious coloring as part of a fixed-parameter tractable algorithm to compute $\h(G)$, parameterized by $\vc(G)$.

Computing the harmonious chromatic number of a graph is a hard task. Given a graph $G$ and a positive integer $k$, determining whether $\h(G)\leq k$ is $\NP$-complete, even if $G$ is a split graph~\cite{asdre2007harmonious}, or an interval graph~\cite{asdre2007harmonious,bodlaender1989achromatic}, or a tree of radius three, or a forest consisting of three trees of radius two~\cite{ED1995}. To address this computational hardness, Section~\ref{sec::programming} is devoted to presenting Integer Linear Programming formulations, Section~\ref{sec::algorithm} proposes two heuristics for the problem, and Section~\ref{sec:5} presents computational results for the formulations and heuristics on random instances.

We introduce three integer linear programming models, two of them based on the representatives formulation presented in~\cite{campelo:04}. 
To the best of our knowledge, this is the first work to present a mathematical programming approach for harmonious coloring. By the way, the first integer-linear programming formulations to compute the line-distinguishing chromatic number were proposed only recently~\cite{oliveira2019coloraccao}. Of course, these formulations do not impose the constraint that the coloring must be proper. 

We emphasize that we do not only add this constraint 
to the formulation presented in~\cite{oliveira2019coloraccao}. We present a new set of $\mathcal{O}(n^3)$ constraints to ensure a line-distinguishing coloring, which improves the $\mathcal{O}(n^4)$ constraints proposed in~\cite{oliveira2019coloraccao}, where $n$ stands for the number of vertices of the input graph.  
We present computational experiments to test our formulations and heuristics on random instances as well as on instances from the second \textit{DIMACS Implementation Challenge}\cite{dimacs}.

\section{Bound on the vertex cover and the maximum degree}\label{sec:2}

In~\cite{kolay2019harmonious}, the authors present the following bound:

\begin{teo}[\!\!\cite{kolay2019harmonious}]
\label{2019teo1}
    For any graph $G$ with $\Delta(G)\geq2,~\Delta(G)+1\leq\h(G)\leq \vc(G)+\Delta(G)(\Delta(G)-1)$.
\end{teo}

And a refined version of it to $d$-degenerate graphs:

\begin{teo}[\!\!\cite{kolay2019harmonious}]
\label{2019teo2}
    If $G$ is a $d$-degenerate graph, then $\Delta(G)+1\leq\h(G)\leq \vc(G)+d(\Delta(G)-1)+\Delta(G)(d-1)$.
\end{teo}

The techniques to prove these theorems are similar. The lower bound is simply the constraint that vertices at distance two should have distinct colors. For the upper bounds, in both theorems, the authors first present a harmonious coloring of the graph using exactly the proposed upper bounds~$+1$~color, and then a trick is applied to save such extra color. This is how they obtain a harmonious coloring on the graph with exactly $\vc(G)+\Delta(G)(\Delta(G)-1)$ colors at the end of~\Cref{2019teo1} (resp $\vc(G)+d(\Delta(G)-1)+\Delta(G)(d-1)$ colors in~\Cref{2019teo2}).

However, since it has no relation to the graph being $d$-degenerate but only to the maximum degree of $G$, in~\Cref{2019teo2}, the trick for saving a color is misapplied, though it works well for~\Cref{2019teo1}.

As a consequence of Theorem~\ref{2019teo2}, the authors present the corollary:

\begin{cor}[\!\!\cite{kolay2019harmonious}]
\label{cor2019}
If $G$ is a forest with at least one edge, then $\Delta(G)+1\leq\h(G)\leq\vc(G)+\Delta(G)-1$.
\end{cor}

Corollary~\ref{cor2019} is also incorrect. For a counter-example to both statements, let $2K_2$ be the simple graph consisting of four vertices and two edges without a common endpoint. 
Note that $\vc(2K_2) = 2$, $\Delta(2K_2) = 1$, and $2K_2$ is 1-degenerate, but $\h(2K_2)=3$. Another possible counter-example is a star $S_k$, for $k\geq 2$, which is a tree on $k+1$ vertices such that only one is not a leaf. Note that $\vc(S_k) = 1$, $\Delta(G) = k$; again, it is a 1-degenerate graph, but $\h(S_k) = k+1$.

When trying to fix Theorem~\ref{2019teo2}, we were able to improve it.

\begin{teo}
        If $G$ is a $d$-degenerate graph, then $\Delta(G)+1\leq\h(G)\leq \vc(G)+d(\Delta(G)-1)+1$.
\end{teo}

\begin{proof}
    The lower bound is the same bound shown by \cite{kolay2019harmonious}. Let $n=|V(G)|$.
    
    To prove the upper bound, we use induction on $n$ to prove a stronger statement: if $S$ is a vertex cover of $G$, then there is a harmonious coloring $c$ of $G$ with at most $|S|+d(\Delta(G)-1)+1$ colors such that each vertex $v$ of $S$ has an \emph{individual color}, i.e. no other vertex in $V(G)\setminus \{v\}$ is colored with color $c(v)$.

    If $G$ is trivial, i.e. $n=1$,  then $G$ is 0-degenerate and $\Delta(G)=0$. So the result follows since, regardless whether $S$ is empty or not, $\h(G)=1 \leq  |S|+d(\Delta(G)-1)+1$.
                              
    Let $G$ be a $d$-degenerate graph with $n\geq 2$ vertices, and let $S$ be a vertex cover of $G$. Since $G$ is $d$-degenerate, there exists an ordering $\sigma=(v_1,\ldots, v_n)$ over $V(G)$ such that any vertex $v_i$ has at most $d$ neighbors with indices smaller than $i$, for every $i\in\{1,\ldots, n\}$. Let $H=G-v_n$. Since $S$ is a vertex cover of $G$, note that $S' = S\setminus\{v_n\}$ is a vertex cover of $H$. Moreover, $H$ is $d$-degenerate. By inductive hypothesis, let $c'$ be a harmonious $(|S'|+d(\Delta(H)-1)+1)$-coloring of $H$ such that the vertices of $S'$ have individual colors.
    Now consider the two following cases:
    \begin{enumerate}[ 1) ]
        \item $v_n\in S$: In this case, $|S| = |S'|+1$. So we can extend $c'$ to $V(G)$ by using a new color only in $v_n$, creating a harmonious $(|S'|+d(\Delta(H)-1)+2)$-coloring $c$ of $G$. Note that: $$|S'|+d(\Delta(H)-1)+2= |S|+d(\Delta(H)-1)+1\leq |S|+d(\Delta(G)-1)+1.$$
        So $G$ can be harmoniously colored with $|S|+d(\Delta(G)-1)+1$ colors, and in $c$ each vertex of $S$ uses an individual color.

        \item $v_n\notin S$: In this case, $S = S'$. Since $S$ is a vertex cover of $G$, then $V(G)\setminus S$ is an independent set of $G$. By inductive hypothesis, each vertex of $S=S'$ is colored with an individual color in $c'$. Note that if we extend $c'$ to a coloring $c$ of $V(G)$ by coloring $v_n$ with a color that does not occur in $S = S'$, nor in the vertices at distance two from $v_n$, then $c$ is a harmonious coloring of $G$. By the choice of $v_n$, $v_n$ has at most $d$ neighbors in $G$, and thus at most $d(\Delta(G)-1)$ vertices at distance two. This is one unit less than the $d(\Delta(G)-1)+1$ colors that do not occur in $S=S'$ in the coloring $c'$. Thus, we can extend $c'$ to a harmonious $(|S|+d(\Delta(G)-1)+1)$-coloring $c$ of $G$ such that the vertices of $S$ have individual colors.
        
    \end{enumerate}
    The result follows by considering a minimum vertex cover $S$, i.e. with $|S| = \vc(G)$.
\end{proof}

\section{Identifications of vertices at distance three}
\label{sec::theory}

Let $G\circ uv$ denote the graph obtained from $G$ by the \emph{identification} of two non-adjacent distinct vertices $u,v\in V(G)$, with the removal of multiple edges in case they appear. Recall that such an operation removes $u$ and $v$, adds a new vertex $w$, and adds an edge from $w$ to each vertex adjacent to $u$ or $v$.

It is well-known that $\chi(G) = \n(G)$ if, and only if, $G$ is complete. Due to Zykov~\cite{zykov1949some}, it is also known that, given $u,v\in V(G)$ and $uv\notin E(G)$, the chromatic number of $G$ satisfies 
\begin{equation}\label{eq:zykov}
    \chi(G)=\min\{\chi(G+uv),\chi(G\circ uv)\}.
\end{equation}
With similar ideas, one can prove that:

\begin{prop}
\label{prop:identification-chromatic-number}
For any graph $G$, the following holds:
\begin{enumerate}[(a)]
    \item\label{itema} if $\chi(G)=k$ then there exists a sequence of identifications of non-adjacent vertices to be applied to $G$ such that the remaining graph is $K_k$; and
    \item\label{itemb} if $K_k$ is obtained from $G$ by a sequence of identifications of non-adjacent vertices, then $\chi(G)\leq \chi(K_k)=k$.
\end{enumerate}
\end{prop}

In the sequel, we present similar observations to the case of Harmonious Colorings. Let us first present a quite simple statement proved by~\cite{kundrik}. Let $\diam(G)$ denote the diameter of $G$, i.e. the maximum distance between two vertices. 

\begin{lema}[\!\!\cite{kundrik}]\label{lemad2}
For any graph $G$, $\h(G)=\n(G)$ if, and only if, $\diam(G)\leq 2$.
\end{lema}

\begin{proof}
Suppose first that $\diam(G)>2$. Then, there are $u,v\in V(G)$ such that $\dist_G(u,v)\geq 3$. Consider that $V(G)=\{u,v,w_1,w_2,\ldots,w_{\n(G)-2}\}$. Let us define a coloring $c:V(G)\to\{1,\ldots, \n(G)-1\}$ as follows: $c(u)=c(v)=1$ and $c(w_i)=i+1$, for every $i\in\{1,\ldots,\n(G)-2\}$. Note that $c$ is a harmonious coloring of $G$ using $\n(G)-1$ colors, and so $\h(G)<n$. The converse is trivial.
\end{proof}

Concerning vertex identifications, let us prove a corresponding version of Proposition~\ref{prop:identification-chromatic-number} to the harmonious chromatic number.

\begin{lema}\label{lemteo1}
    Let $G$ be an graph such that $\diam(G)\geq 3$. Then,
    \begin{enumerate}[(a)]
        \item\label{item1} there is a pair $u,v\in V(G)$ such that $\dist_G(u,v)\geq 3$ whose identification results in a graph $H$ with $\h(H)=\h(G)$; and
        \item\label{item2} if $H$ is obtained from $G$ by the identification of $u,v\in V(G)$ such that $\dist_G(u,v)\geq 3$, then $\h(G)\leq \h(H)$.
    \end{enumerate}
\end{lema}

\begin{proof}
    Consider $|V(G)|=n$.
    Since $\diam(G)\geq 3$, by Lemma~\ref{lemad2} we deduce that $\h(G)=k\leq n(G)-1$. Hence, in every harmonious $k$-coloring of $G$ there is a pair of distinct vertices in the same color class. Let $c\colon V(G)\to \{1, \ldots, k\}$ be a harmonious $k$-coloring of $G$, and let $u,v\in V(G)$ be such that $c(u)=c(v)$. Since $c$ is a harmonious coloring, we have that $\dist_G(u,v)\geq 3$. Consider $V(G)=\{u,v,w_1,\ldots,w_{n-2}\}$ and let $H = G\circ uv$ be the graph obtained from the identification of $u$ and $v$. Denote by $w$ the vertex resulting from this operation. So, $V(H)=\{w,w_1,\ldots,w_{n-2}\}$.
    
    Define the coloring $c'$ of $H$ as $c'(w_i)=c(w_i)$ for each $1\leq i\leq n-2$ and $c'(w)=c(v)$. Since we obtain $w$ from the identification of $u$ and $v$ and $c$ is a harmonious coloring of $G$ such that $c(u)=c(v)$, $c'$ is a harmonious coloring of $H$. Therefore, $\h(H)\leq k=\h(G)$. The proof of (\ref{item1}) is completed with the proof of (\ref{item2}) that follows.

    Let $c'$ be a harmonious $p$-coloring of $H$, $p=\h(H)$. Define $c\colon V(G)\to \{1,\ldots, p\}$ such that $c(w_i)=c'(w_i)$ for each $1\leq i\leq n-2$, and $c(u)=c(v)=c(w)$. We will show that $c$ is a harmonious $p$-coloring of $G$. By the construction of $H$ and knowing that $uv\notin E(G)$, certainly $c$ is a proper coloring of $G$. Since $c'$ is a harmonious coloring, if there are two edges $e_1,e_2\in E(G)$ with the same colors in their endpoints, then certainly $v$ is an endpoint of one of these edges and $u$ is an endpoint of the other edge. Without lost of generality, consider $e_1=uz$, $e_2=vy$, $c(u)=c(v)=1$ and $c(z)=c(y)=2$. However, $u\text{ and }v$ are identified to obtain $H$, so $\dist_H(z,y)\leq 2$. Thus, because $c'$ is a harmonious coloring, the equality $c(z)=c(y)=c'(z)=c'(y)$ is not possible. 
\end{proof}

One may naturally deduce what happens after applying Lemma~\ref{lemteo1} iteratively:

\begin{teo}\label{idlema}
For any graph $G$, the following hold:
\begin{enumerate}[(a)]
    \item\label{item3} if $\h(G)=k$, then there exists a sequence of identifications of vertices to be applied to $G$ such that, at each step, the identified vertices have distance at least 3 and such that the obtained graph $H$ in the end has $k$ vertices and diameter at most 2; and
    \item\label{item4} if $H$ is obtained from $G$ by a sequence of identifications of vertices such that, at each step, the identified vertices have distance at least 3, then $\h(G)\leq \h(H)$.
\end{enumerate}
\end{teo}
\begin{proof}
To prove~(\ref{item1}), suppose first that $\h(G) = k$. If $\diam(G)\leq 2$, then $\h(G)=\n(G)$ and one may take an empty sequence. Suppose now that $\diam(G)\geq 3$. Thus, it is enough to apply Lemma~\ref{lemteo1}(\ref{item1}) exhaustively until the resulting graph has diameter at most 2, and then we apply Lemma~\ref{lemad2}. 

The second statement follows by successive applications of Lemma~\ref{lemteo1}(\ref{item2}).
\end{proof} 

Theorem~\ref{idlema} is our source of inspiration for a heuristic we propose to obtain feasible harmonious colorings in Section~\ref{sec::algorithm}. We must also emphasize that Statement~(\ref{item3}) of Theorem~\ref{idlema} is proved in~\cite{kolay2019harmonious}, with the difference that they identify all vertices in the same color class of an optimal harmonious coloring of $G$ at once.

Another remark is that Zykov's equation~\eqref{eq:zykov} to the chromatic number does not translate easily as adding edges to $G$ not only creates new constraints to obtain a proper coloring, but also to obtain a harmonious coloring. A natural question is:

\begin{problem}
    Can one find an equation similar to Zykov's to the harmonious chromatic number?
\end{problem}

\section{Integer Programming Formulations}
\label{sec::programming}
Without loss of generality, we assume that the input graph $G$ has no isolated vertices, as such a vertex would not alter the harmonious chromatic number. This will simplify the notation and the presentation of the formulations in this section.

\subsection{Standard Model}

Given a graph $G=(V,E)$ and an upper bound $k$ on $\h(G)$, we define a binary variable $x_{v,i}$, for each $v\in V$ and $1\leq i \leq k$, to indicate whether or not the vertex $v$ receives the color $i$, as well as a binary variable $w_i$, for $1\leq i \leq k$, representing whether or not the color $i$ is used. Similar variables are used in a well-known formulation for the classical vertex coloring problem~\cite{Coll02Facets}. Also, for $u,v\in V$, let $N(u)$ stand for the set of vertices that are adjacent to $u$, and $N(u,v)=N(u)\cup N(v)$. Note that if 
$u$ and $v$ are neighbors, both belong to $N(u,v)$. Then, one can provide a simple formulation for the harmonious coloring problem as follows. 
\begin{mini!}
{w,x}{\sum_{i=1}^{k} w_{i} \label{std:obj} }{}{(STD)}
\addConstraint{\sum_{i=1}^{k} x_{v,i}}{= 1}{ \forall v \in V \label{std:color}}
\addConstraint{x_{v,i}+x_{u,i}}{\leq w_{i}}{ \forall i \in [k], \forall uv \in E \label{std:edge}}
\addConstraint{\sum\limits_{v \in N(u)}x_{v,i}}{\leq w_i}{\forall i \in [k], \forall u \in V \label{std:harm1}}
\addConstraint{\sum\limits_{w \in N(u,v)}x_{w,j} }{\leq 1 + w_j + w_i - x_{v,i} - x_{u,i} \quad}
{\forall uv \notin E, u\neq v, \forall i,j \in [k]\label{std:harm2}}
\addConstraint{x_{v,i}, w_{i}}{\in \{0,1\}}{\forall v \in V, \forall i \in [k] \label{std:limits}}
\end{mini!}

The objective function \eqref{std:obj} is what one might expect. It minimizes the number of used colors, represented by the sum of variables $w_i$.
Constraints~\eqref{std:color} ensure that each vertex must have exactly one color.
Constraints~\eqref{std:edge} force the coloring to be proper, i.e. two adjacent vertices must be assigned distinct colors. Since there is no isolated vertex, they also allow a vertex to receive a certain color only if such a color is marked as used.
Constraints~\eqref{std:harm1} guarantee the line-distinguishing property for two edges with a common endpoint, as they prevent two non-adjacent vertices that share a common neighbor from being assigned the same color.
Constraints~\eqref{std:harm2} extend this property to the remaining pairs of edges. Indeed, if two non-adjacent vertices $u$ and $v$ are colored with the same color $i$, these constraints forbid the repetition of any color $j$ in the union of the neighborhoods of $u$ and $v$: the left-hand side must be zero (if $w_j=0$) or at most 1 (if $w_j=1$). Notice that if $u$ or $v$ is not colored with the color $i$, the constraint also works since its right-hand side is at least $1+w_j$, thus allowing this number of vertices with the color $j$ in $N(u,v)$. Finally, 
Constraints~\eqref{std:limits} represent the variables' domain. 

\subsection{Representatives Formulation}

Since a harmonious $k$-coloring of a graph $G$ is a proper coloring, it induces a partition $\{S_1, \dots, S_k\}$ of the vertex set $V(G)$ into $k$ independent sets of $G$.
Suppose that, for each color $i$, we choose a vertex in $S_i$ to be the representative of the color class $S_i$. Then the vertices can be in one of two states: it represents its own color or there exists another vertex that represents its color~\cite{campelo:04}. Notice that each vertex $v$ can only be represented by itself and by vertices that are non-adjacent nor share a common neighbor, i.e. vertices at distance at least three from $v$. Hence, for every $v \in V$, the set of its possible representatives (besides itself) is
$$N^*(v) = \{ u \mid u\notin N(v)\text{ and }N(v) \cap N(u) = \emptyset\} = \{u\in V(G)\mid \dist_G(u,v)\geq 3\},$$
where $\dist_G(u,v)$ is the distance between $u$ and $v$ in $G$.
Let $N^*[v] = N^*(v) \cup \{v\}$. Observe that $u$ can represent $v$ if and only if $v$ can represent $u$.

To model the representative's strategy, we define the binary variable $x_{uv}$ for every $u \in V(G)$ and $v \in N^*[u]$, indicating whether the vertex $u$ is the representative of the color assigned to vertex $v$. With such variables, we can define the following formulation.
\begin{mini!}
{x}{\sum_{u \in V} x_{uu} \label{rep:obj} }{}{(REP)}
\addConstraint{\sum_{u \in N^*[v]} x_{uv}}{= 1}{ v \in V \label{rep:rep}}
\addConstraint{x_{uv} + x_{uw}}{\leq x_{uu}}{ \forall vw\in E, ~\forall u \in N^*(v)\cap N^*(w) \label{rep:edgeG2}}
\addConstraint{\sum_{z \in N^{*}(u) \cap N(v)} x_{uz}}{\leq x_{uu}}{\forall v,u \in V \label{rep:neg}}
\addConstraint{\sum\limits_{\ell \in N^*(w) \cap N(u, v)} \!\!\!\!\!\!\!\!\!\! x_{w\ell} }{ \leq 1 + x_{ww}+x_{zz} -x_{zu} -x_{zv}\quad}{\begin{array}{l}
    \forall w\in V, \forall uv \notin E, u\neq v, \\
     \forall z \in N^*(u)\cap N^*(v),~ z\neq w 
\end{array}\label{rep:harm}}
\addConstraint{x_{uv}}{\in \{0,1\}}{\forall u \in V,  \forall v \in  N^{*}[u] \label{rep:limits}}
\end{mini!}

The expressions~\eqref{rep:obj}-\eqref{rep:harm} are the counterparts of \eqref{std:obj}-\eqref{std:harm2} with variables $x_{uu}$ and $x_{uv}$ playing the role of $w_i$ and $w_{v,i}$, respectively. Indeed, the objective function~\eqref{rep:obj} minimizes the number of representative vertices since each representative counts as a different color. Constraints~\eqref{rep:rep} indicate that each vertex $v \in V(G)$ must be represented either by itself or by some other vertex in $N^*(v)$. 
Constraints~\eqref{rep:edgeG2} ensure that each pair of adjacent vertices that could be represented by the same vertex must be assigned distinct colors. Since there is no isolated vertex, they also enforce the variable $x_{uu}$ to be one, i.e. that $u$ must be the representative of its color, whenever a vertex $v$ has $u$ as representative of its color.
Constraints~\eqref{rep:neg} ensure that at most one vertex $z$ in the neighborhood of a vertex $v$ can be represented by a vertex $u$, which must be at distance at least three from $z$.
Given two non-adjacent vertices $u$ and $v$, Constraints~\eqref{rep:harm} indicate that if they have the same color, i.e., they have the same representative $z$, there must be at most one vertex in their conjoint neighborhood that has the same color represented by $w$; also, notice that if $u$ and $v$ do not share the same color, then in their conjoint neighborhood, any color appears at most twice, thus ensuring the validity of the constraints in this case. Finally, Constraints~\eqref{rep:limits} are the definition of the variable's bounds.

\subsection{Asymmetric Representatives Formulation}

The representatives formulation has some symmetric solutions given by the possibility of choosing any vertex in a color class to be the representative vertex of that class. To tackle this problem, as in \cite{campelo08asymetric}, we consider an arbitrary total ordering $\leq$ over the vertex set, and take the representative of a color $i$ as the minimum, with respect to $\leq$, vertex with color $i$. This way any coloring has a unique description in terms of representative vertices.

Thus, for every $v\in V(G)$, let us partition $N^*(v)$ into $N_{+}^*(v) = \{ u \in N^*(v):u>v\}$ and $N_{-}^*(v) = \{ u \in N^*(v):u< v\}$. These sets comprise the vertices (other than $v$) that can be represented by $v$ and that can represent $v$, respectively. Besides, let $N_{-}^*[v] =N_{-}^*(v)\cup \{v\}$. Also, consider $S$ as the set comprising every vertex whose unique possible representative is itself.
More formally, $S =\{ v \in V(G) \mid N_{-}^*(v) = \emptyset\}$.

In addition to avoiding symmetrical solutions, the asymmetric representation also demands fewer variables than those used in the representatives formulation. Indeed, we only define the binary variable $x_{uv}$, for every $u \in V(G)$ and $v \in N^{*}_{+}(u)$, to indicate whether or not vertex $u$ is the representative of vertex $v$, as well as $x_{uu}$, for every $u \in V(G) \setminus S$, meaning whether or not $u$ is a representative vertex. Notice that there is no need to create variables $x_{uu}$ for the vertices in $S$, as they will certainly be representatives. Instead, the corresponding variables can be fixed at 1 a priori.
To keep the formulation presentation as simple as possible, we introduce the notation $y_{u}$ for each $u \in V(G)$ standing for $x_{uu}$, if $u \not\in S$, or $1$, otherwise. Then, we can define the following formulation.

\begin{mini!}
{x}{\sum_{u \in V} x_{uu} \label{arep:obj} }{}{(AREP)}
\addConstraint{\sum\limits_{u \in N^{*}_{-}[v]} x_{uv}}{= 1}{ \forall v \in V \setminus S \label{arep:rep}}
\addConstraint{x_{uv}}{\leq y_{u}}{\forall u \in V,~  \forall v \in N^{*}_{+}(u) \label{arep:rep2}}
\addConstraint{x_{uv} + x_{uw}}{\leq y_{u}}{\forall vw \in E, \forall u \in N^{*}_{-}(v)\cap N^{*}_{-}(w) \label{arep:edgeG2}}
\addConstraint{\sum_{z \in N^{*}_{+}(u) \cap N(v)} x_{uz}}{\leq y_{u}}{\forall u,v \in V \label{arep:neg}}
\addConstraint{\sum\limits_{\ell \in N^*_+(w) \cap N(u, v)} \!\!\!\!\!\!\!\!\!\! x_{w\ell} }{ \leq 1 + y_{w}+y_{z} -x_{zu} -x_{zv}\quad}{\begin{array}{l}
    \forall w\in V, \forall uv \notin E, u\neq v, \\
     \forall z \in N^*_-(u)\cap N^*_-(v),~ z\neq w 
\end{array}\label{arep:harm}}
\addConstraint{x_{uv}}{\in \{0,1\}}~{\forall v \in V,~ \forall u \in N^{*}_{-}(v)} \label{arep:limits1}
\addConstraint{x_{uu}}{\in \{0,1\}}~{\forall u \in V \setminus S} \label{arep:limits2}
\end{mini!}

The above formulation is essentially \eqref{rep:obj}-\eqref{rep:limits} with $N^*(v)$ appropriately replaced by $N^*_+(v)$ or $N^*_-(v)$ depending on whether one would mean the vertices represented by $v$ or those that $v$ can represent. Note that \eqref{arep:rep2} is only necessary if $v$ is isolated in $G[N^{*}_{+}(u)]$.

\section{Greedy Heuristics}
\label{sec::algorithm}
As computing the harmonious chromatic number is an NP-Hard problem, one has little to no hope of solving large instances of the problem exactly. Hence, we propose heuristics to handle such large instances. A simple type of heuristic method that commonly represents the first attempt to tackle any problem is the greedy heuristic. 

We present a simple greedy heuristic based on the iterative process of identification of vertices discussed in Section~\ref{sec::theory}. We propose two different greedy choices for the pair to be contracted: the minimum degree sum (Algorithm~\ref{alg:heuristicmin}) and the maximum degree sum (Algorithm~\ref{alg:heuristicmax}). Both methods use the result presented in Theorem~\ref{idlema} to compute the correct value of a harmonic coloring that colors the identified vertices with the same color.

\begin{algorithm}[H]
\caption {Greedy MIN ($G$)}
\label{alg:heuristicmin}
    $H \leftarrow G$\;
    \While{$H$ has $v$ and $u$ with $\dist_G(u,v)\geq 3$}{
        Select $u$ and $v$ such that $N(v) \cap N(u) = \emptyset$ and $d_G(u) + d_G(v)$ is minimum among all possible pairs\;
        Identify vertices $u$ and $v$ in $H$\;
    }
    \Return $|V(H)|$\;
\end{algorithm}

\begin{algorithm}[H]
\caption {Greedy MAX ($G$)}
\label{alg:heuristicmax}
    $H \leftarrow G$\;
    \While{$H$ has $v$ and $u$ with $\dist_G(u,v)\geq 3$}{
        Select $u$ and $v$ such that $N(v) \cap N(u) = \emptyset$ and $d_G(u) + d_G(v)$ is maximum among all possible pairs\;
        Identify vertices $u$ and $v$ in $H$\;
    }
    \Return $|V(H)|$\;
\end{algorithm}

Both heuristics aim to minimize the maximum degree of the graph obtained by the sequence of identifications. While Algorithm~\ref{alg:heuristicmax} tries to attain this goal by minimizing 
the total number of identifications (as a consequence of prioritizing pairs of high-degree vertices), Algorithm~\ref{alg:heuristicmin} allows a higher number of identifications but increases the maximum degree slowly (by using small degree vertices).

\section{Computational Results}\label{sec:5}
\label{sec::results}

\subsection{Environment and Instances Description}

All computational experiments were executed on a machine running Ubuntu 20.04 x86-64 GNU/Linux, with an Intel Core i7-7500 Octa-Core 2.70 GHz processor and 20 GB of RAM. The ILP formulations were implemented in Julia and solved with Gurobi 10.0.1. The greedy heuristics were also coded in Julia. 

The benchmark instances used in the computational experiments were divided into two sets: {\bf random} and {\bf dimacs}. Both sets are also divided into {\em small} and {\em large} instances.

The set {\bf random} is composed of randomly generated graphs, that respect a given probability of having an edge between any distinct pair of vertices. More precisely, for a given natural $n$ and a probability $p$, we generate a graph with $n$ vertices and add an edge between every pair of distinct vertices independently with probability $p$. Note that $p$ is an estimation of the edge density of the graph. Besides, such a process is inherently random and can produce different structures. Thus, for each pair composed of a vertex number $n\in \{$10,20,30,40,50,60,70,80,90,100$\}$ and an edge probability $p\in \{$0.05,0.10,0.20,0.30,0.40$\}$,  we generate 5 distinct graphs.
Thus, the set \textbf{random} has 250 instances. The subset of {\em small} instances comprises those where $n \leq 50$ and $0.20 \leq p$. 
The set \textbf{dimacs}\cite{instances} comprises 31 graphs that are challenging for the classical vertex coloring problem. This set encompasses diverse graph structures that are known to yield challenging instances for coloring problems. We also select a part of these instances to form a  set of {\em small} instances, but the criterion adopted is only the number of vertices, due to the diversity of structures of the dimacs graphs. More precisely, we select all instances that have 100 vertices or fewer.

Recall that the square graph $G^2$ of a graph $G$ is obtained from $G$ by adding an edge between non-adjacent vertices at distance of 2. This means that a harmonious coloring of $G$ is a conventional proper coloring of $G^2$ with additional restrictions. 
Table~\ref{tab:random:summary2} presents the true edge density $d$ of $G$ and $G^2$ related to the instances in the set \emph{random}. Each percentage is an average of the densities of the corresponding 5 instances. We remark that every graph built with parameter $p \geq 0.50$ has a complete graph as its square graph; hence, its harmonic chromatic number is equal to the number of vertices, and therefore one does not need to use an integer linear program to obtain such a number.

\begin{table}[!h]
\begin{center}	
\begin{scriptsize}
\begin{tabular}{|c|c|c|c|c|c|c|c|c|c|c|}
\hline & \multicolumn{2}{|c|}{ $p=0.05$ } & \multicolumn{2}{|c|}{ $p=0.10$ } & \multicolumn{2}{|c|}{ $p=0.20$ } & \multicolumn{2}{|c|}{ $p=0.30$ } & \multicolumn{2}{|c|}{ $p=0.40$ }\\
\hline $n(V)$ & $d(G)$ & $d(G^2)$ & $d(G)$ & $d(G^2)$ & $d(G)$ & $d(G^2)$ & $d(G)$ & $d(G^2)$ & $d(G)$ & $d(G^2)$ \\ \hline 
10 & 4.7\% & 9.5\% & 10.9\% & 25.8\% & 15.6\% & 36.0\% & 26.5\% & 61.8\% & 31.6\% & 79.6\%\\ 
20 & 5.0\% & 11.1\% & 9.8\% & 25.9\% & 18.1\% & 59.6\% & 27.2\% & 82.2\% & 37.5\% & 97.2\%\\ 
30 & 6.6\% & 17.9\% & 10.2\% & 32.7\% & 18.0\% & 67.1\% & 27.8\% & 92.1\% & 37.7\% & 99.1\%\\ 
40 & 5.7\% & 16.0\% & 10.3\% & 39.0\% & 18.1\% & 76.4\% & 28.7\% & 97.5\% & 38.4\% & 99.9\%\\ 
50 & 6.1\% & 21.2\% & 10.9\% & 49.0\% & 18.2\% & 83.1\% & 28.8\% & 99.0\% & 38.1\% & 100.0\%\\ 
60 & 5.7\% & 21.5\% & 10.8\% & 54.2\% & 19.6\% & 92.0\% & 28.6\% & 99.5\% & 39.3\% & 100.0\%\\ 
70 & 5.6\% & 22.4\% & 10.6\% & 56.8\% & 18.9\% & 93.2\% & 29.7\% & 99.9\% & 38.6\% & 100.0\%\\ 
80 & 5.9\% & 26.1\% & 10.6\% & 60.7\% & 19.2\% & 95.7\% & 29.1\% & 99.9\% & 38.3\% & 100.0\%\\ 
90 & 6.0\% & 30.0\% & 10.9\% & 67.8\% & 19.3\% & 97.4\% & 29.6\% & 100.0\% & 38.8\% & 100.0\%\\ 
100 & 5.9\% & 31.2\% & 11.2\% & 73.0\% & 19.9\% & 98.2\% & 29.0\% & 100.0\% & 39.7\% & 100.0\%\\ 
\hline 
\end{tabular}
\caption{Mean density values for the instances in the test set {\bf random}, grouped by number of vertices and value of parameter $p$.}
\label{tab:random:summary2}
\end{scriptsize}
\end{center}		
\end{table} 

Table~\ref{tab:dimacs:summary} presents the same summary for the instances in the set {\bf dimacs}. We remark that the square of 11 of these graphs is a complete graph. 
So they were not used as test instances. Furthermore, the majority of the other square graphs have high density, above $80\%$.

\begin{table}[!h]
\begin{center}	
\begin{scriptsize}
\begin{tabular}{|l|r|r|r|r|}
\hline name & $|V(G)|$ & $|E(G)|$ & $d(G)$ & $d(G^2)$ \\ \hline 
myciel4.col & 23 & 71 & 25.7\% & 74.3\%\\ 
myciel5.col & 47 & 236 & 20.9\% & 79.1\%\\ 
david.col & 87 & 406 & 10.6\% & 94.5\%\\ 
queen13\_13.col & 169 & 3328 & 23.2\% & {\bf 100.0\%}\\ 
anna.col & 138 & 493 & 5.1\% & 54.7\%\\ 
games120.col & 120 & 638 & 8.8\% & 40.2\%\\ 
queen14\_14.col & 196 & 4186 & 21.7\% & {\bf 100.0\%}\\ 
miles250.col & 128 & 387 & 4.7\% & 12.1\%\\ 
queen5\_5.col & 25 & 160 & 49.2\% & {\bf 100.0\%}\\ 
queen8\_8.col & 64 & 728 & 35.0\% & {\bf 100.0\%}\\ 
queen12\_12.col & 144 & 2596 & 24.9\% & {\bf 100.0\%}\\ 
mulsol.i.1.col & 197 & 3925 & 20.1\% & 49.1\%\\ 
miles500.col & 128 & 1170 & 14.2\% & 34.5\%\\ 
huck.col & 74 & 301 & 10.8\% & 65.0\%\\ 
mulsol.i.4.col & 185 & 3946 & 22.9\% & 89.4\%\\ 
mulsol.i.2.col & 188 & 3885 & 21.9\% & 84.6\%\\ 
mulsol.i.3.col & 184 & 3916 & 23.0\% & 89.4\%\\ 
mulsol.i.5.col & 186 & 3973 & 22.8\% & 89.5\%\\ 
miles1500.col & 128 & 5198 & 63.0\% & 95.4\%\\ 
myciel7.col & 191 & 2360 & 12.9\% & 87.1\%\\ 
queen7\_7.col & 49 & 476 & 38.9\% & {\bf 100.0\%}\\ 
miles750.col & 128 & 2113 & 25.6\% & 60.4\%\\ 
miles1000.col & 128 & 3216 & 39.0\% & 75.5\%\\ 
jean.col & 80 & 254 & 7.8\% & 40.2\%\\ 
queen9\_9.col & 81 & 1056 & 31.8\% & {\bf 100.0\%}\\ 
queen8\_12.col & 96 & 1368 & 29.4\% & {\bf 100.0\%}\\ 
queen10\_10.col & 100 & 1470 & 29.1\% & {\bf 100.0\%}\\ 
queen11\_11.col & 121 & 1980 & 26.8\% & {\bf 100.0\%}\\ 
queen6\_6.col & 36 & 290 & 43.5\% & {\bf 100.0\%}\\ 
myciel6.col & 95 & 755 & 16.6\% & 83.4\%\\ 
myciel3.col & 11 & 20 & 30.3\% & 69.7\%\\ 
\hline 
\end{tabular}
\caption{Density values for the instances in the test set {\bf dimacs}.}
\label{tab:dimacs:summary}
\end{scriptsize}
\end{center}		
\end{table}

\subsection{Comparison of the Mathematical Models}

The first group of experiments was designed to assess the effectiveness of the proposed mathematical models. Therefore, we used the small instances in such experiments. We limited the execution time of the solver to 1800 seconds in this first set of experiments, used a single thread, and left the standard configuration of all the solver's parameters. 

All the presented formulations have a group of constraints to define stable sets (color classes): \eqref{std:edge} in the standard formulation, \eqref{rep:edgeG2} in the representatives formulation and \eqref{arep:edgeG2} in the asymmetric formulation. Although correct, this kind of constraint tends to enlarge the model and is commonly replaced by a set of constraints related to a clique cover of the edges of the graph. More precisely, assume that $K_1, K_2, \ldots ,K_p$ is a set of 
maximal cliques in the input graph square ($G^{2}$) such that the endpoints of every edge belong to at least one clique. 
Then, for each clique $K$, the $|K|(|K|-1)/2$ constraints in each formulation forbidding both endpoints of every edge in $K$ from receiving the same color can be replaced by a single constraint ensuring that at most one vertex in $K$ is assigned any given color.
This approach reduces the number of constraints in each model and leads to better computational results. Thus, we employ such a strategy while keeping the same choice of cliques for the three formulations to preserve the comparison between them.

The maximal clique cover we use in the experiments is built as follows. The first step selects a vertex $u$ with maximum degree in the graph square and creates a set $K$ containing only $u$. Then, we iteratively select a vertex of maximum degree in $G^2[N(K)\setminus K]$ and add such a vertex to $K$. Eventually, $K$ becomes a maximal clique in $G^2$, and it is used in the clique cover. At this point, all edges from $G^2[K]$ are discarded, and the iterative process is repeated.  This procedure is formalized in Algorithm~\ref{alg:maxcliquecover}, which is applied with $H=G^2$.

\begin{algorithm}[H]
\caption {Maximal Clique Cover ($H$)}
\label{alg:maxcliquecover}
    $\mathcal{K} \leftarrow \emptyset$\;
    \While{$E(H) \not= \emptyset$}{
        Select $u$ with maximal degree in $H$\;
        $K \leftarrow \{u\}$\;
        $cand \leftarrow N(u)$\;
        \While{$cand \not= \emptyset$}{
            Select $v$ with maximal degree in $H[cand]$\;
            $K \leftarrow K \cup \{v\}$\;
            $cand \leftarrow cand \cap N(v)$\;
        }
        $\mathcal{K} \leftarrow \mathcal{K} + K$\;
        $E(H) \leftarrow E(H) - E(H[K])$
    }
    \Return $\mathcal{K}$\;
\end{algorithm}

Finally, before presenting the computational results, we clarify two specific points related to the mathematical models: the number of colors used in $STD$ and the vertex ordering used in $REP$ and $AREP$. For the first issue, we use the minimum between the upper bound presented in \cite{HK1983} and the number of vertices of the graph. For the second one, we adopt a non-decreasing degree order of the vertices, similarly to~\cite{campelo08asymetric}. We note that both choices could be further improved, but this would require a more in-depth study of the models, which is beyond the scope of this paper.

The results are reported in Tables~\ref{tab:random:opt} and~\ref{tab:dim:opt} for the small instances in the set {\bf random} and in the set {\bf dimacs}, respectively. A group of columns is associated with each formulation: the standard formulation ({\tt STD}), the representatives formulation ({\tt REP}) and the asymmetric representatives formulation ({\tt AREP}). Each group has three columns: {\tt BSOL} - the best solution known at the end of the time limit (if no solution is found, the entry receives the mark {\bf -}); {\tt GAP} - the relative gap as reported by the solver; and {\tt time} - the elapsed time as reported by the solver. In Table~\ref{tab:random:opt}, each entry is the average of instances with the same number of vertices and edge assignment probability.
As we present the results grouped, we should explain what happens when one or more instances of the group reach the time limit or do not provide a feasible solution or bound. 
Instances for which no feasible solution or bound is found are disregarded when computing the reported means.As for the instances that reached the time limit but yielded feasible solutions, we report their information as presented by the solver and use the time limit proposed for the experiments (1800 seconds) as the elapsed time. If all 5 instances reached the time limit, the mark {\tt limit} is shown in the corresponding entry. 

Regarding the set \textbf{random}, notice that both representatives formulations outperform the standard formulation, being able to solve to optimality almost all instances. Between the two representatives formulations, the asymmetric version ({\tt AREP}) was able to solve to optimality all instances solved by the symmetric version ({\tt REP}), as well as additional ones, while consistently requiring less elapsed time and yielding smaller gaps.
The differences between the two representatives formulations are more pronounced when $p=0.05$ and $0.10$. For 30 vertices, ({\tt AREP}) could optimally solve these sparser instances while ({\tt REP}) was not able to close the optimality gap. For vertex number equal to 40 and 50, {\tt AREP} significantly reduced the gap in these instances instances. Besides, {\tt AREP} consistently took less time.

\begin{table}[!h]
\begin{center}
\begin{scriptsize}
\begin{tabular}{cc|rrr|rrr|rrr}
 & & \multicolumn{3}{|c|}{ STD } & \multicolumn{3}{|c|}{ REP } & \multicolumn{3}{|c}{ AREP } \\ 
\hline $|V(G)|$ & $p$ & BSOL & GAP & TIME & BSOL & GAP & TIME & BSOL & GAP & TIME \\ 
\hline10 & 0.05 & 3.0 & 0.0\% & 0.006 & 3.0 & 0.0\% & 0.019 & 3.0 & 0.0\% & 0.001 \\ 
 10 & 0.10 & 4.8 & 0.0\% & 0.012 & 4.8 & 0.0\% & 0.024 & 4.8 & 0.0\% & 0.001 \\ 
 10 & 0.20 & 5.2 & 0.0\% & 0.046 & 5.2 & 0.0\% & 0.016 & 5.2 & 0.0\% & 0.002 \\ 
 10 & 0.30 & 7.2 & 0.0\% & 0.123 & 7.2 & 0.0\% & 0.003 & 7.2 & 0.0\% & 0.001 \\ 
 10 & 0.40 & 7.8 & 0.0\% & 0.084 & 7.8 & 0.0\% & 0.001 & 7.8 & 0.0\% & 0.001 \\ 
 20 & 0.05 & 5.6 & 0.0\% & 0.448 & 5.6 & 0.0\% & 33.060 & 5.6 & 0.0\% & 0.026 \\ 
 20 & 0.10 & 7.4 & 0.0\% & 3.423 & 7.4 & 0.0\% & 49.308 & 7.4 & 0.0\% & 0.035 \\ 
 20 & 0.20 & 10.8 & 1.818\% & 783.795 & 10.8 & 0.0\% & 1.649 & 10.8 & 0.0\% & 0.032 \\ 
 20 & 0.30 & 14.4 & 0.0\% & 99.365 & 14.4 & 0.0\% & 0.051 & 14.4 & 0.0\% & 0.016 \\ 
 20 & 0.40 & 19.0 & 0.0\% & 10.94 & 19.0 & 0.0\% & 0.004 & 19.0 & 0.0\% & 0.003 \\ 
 30 & 0.05 & 9.0 & 6.667\% & 1680.506 & 9.0 & 29.000\% & {\tt limit} & 9.0 & 0.0\% & 16.436 \\ 
 30 & 0.10 & 11.0 & 12.455\% & {\tt limit} & 11.0 & 19.909\% & {\tt limit} & 11.0 & 0.0\% & 1.782 \\ 
 30 & 0.20 & 16.4 & 11.168\% & {\tt limit} & 16.4 & 0.0\% & 121.619 & 16.4 & 0.0\% & 0.430 \\ 
 30 & 0.30 & 23.4 & 0.909\% & 1648.509 & 23.4 & 0.0\% & 0.150 & 23.4 & 0.0\% & 0.037 \\ 
 30 & 0.40 & 29.2 & 0.0\% & 665.829 & 29.2 & 0.0\% & 0.008 & 29.2 & 0.0\% & 0.006 \\ 
 40 & 0.05 & 10.6 & 24.727\% & {\tt limit} & 11.0 & 38.121\% & 408.927 & 10.6 & 5.636\% & {\tt limit} \\ 
 40 & 0.10 & 13.0 & 23.077\% & {\tt limit} & 15.4 & 34.811\% & 1161.243 & 14.6 & 1.333\% & 1188.028 \\ 
 40 & 0.20 &  -  &  -  & {\tt limit} & 23.0 & 6.000\% & 167.749 & 22.8 & 0.0\% & 1.828 \\ 
 40 & 0.30 &  -  &  -  & {\tt limit} & 34.6 & 0.0\% & 0.184 & 34.6 & 0.0\% & 0.048 \\ 
 40 & 0.40 &  -  &  -  & {\tt limit} & 39.6 & 0.0\% & 0.014 & 39.6 & 0.0\% & 0.012 \\ 
 50 & 0.05 &  -  &  -  & {\tt limit} & 14.7 & 38.793\% & {\tt limit} & 14.0 & 24.139\% & {\tt limit} \\ 
 50 & 0.10 &  -  &  -  & {\tt limit} & 21.8 & 36.797\% & 666.415 & 19.5 & 9.722\% & 391.382 \\ 
 50 & 0.20 &  -  &  -  & {\tt limit} & 29.6 & 0.0\% & 121.581 & 29.6 & 0.0\% & 3.412 \\ 
 50 & 0.30 &  -  &  -  & {\tt limit} & 45.2 & 0.0\% & 0.230 & 45.2 & 0.0\% & 0.065 \\ 
 50 & 0.40 &  -  &  -  & {\tt limit} & 49.6 & 0.0\% & 0.037 & 49.6 & 0.0\% & 0.035 \\ 
 \hline
 \end{tabular}
 \caption{Results obtained by the three proposed formulations on small instances of the benchmark set {\bf random}.
 }
 \label{tab:random:opt}
\end{scriptsize} 
\end{center}
\end{table}

The conclusions for the small instances in the set \textbf{dimacs} are similar. In Table~\ref{tab:dim:opt}, we only present the results for the instances that had a solution found by at least one of the formulations. 
Again, both representatives formulations outperform the standard formulation, which reaches the time limit for the majority of the instances in the set. As before, the asymmetric version is consistently faster than the symmetric version. It was even able to solve one instance in which {\tt REP} failed to find a feasible solution, namely {\tt jean.col}.

\begin{table}[!h]
\begin{scriptsize}
\begin{center}
\begin{tabular}{l|rrr|rrr|rrr}
&  \multicolumn{3}{|c|}{ STD } & \multicolumn{3}{|c|}{ REP } & \multicolumn{3}{|c}{ AREP } \\ 
\hline Instance & BSOL & GAP & TIME & BSOL & GAP & TIME & BSOL & GAP & TIME \\ 
\hline
david.col &
- & - & {\tt limit} &
83.0 & 0\% & 9.407 &
83.0 & 0\% & 1.714 \\  
huck.col &
- & - & {\tt limit} &
54.0 & 0\% & 31.773 &
54.0 & 0\% & 9.426 \\
jean.col &
- & - & {\tt limit} &
- & - & {\tt limit} &
37.0 & 0\% & 21.455 \\
myciel3.col &
11.0 & 0\% & 0.495 &
11.0 & 0\% & $<$0.001 &
11.0 & 0\% & $<$0.001 \\
myciel4.col &
23.0 & 0\% & 215.979 &
23.0 & 0\% & 0.007 &
23.0 & 0\% & 0.007 \\
myciel5.col &
- & - & {\tt limit} &
47.0 & 0\% & 0.016 &
47.0 & 0\% & 0.013 \\
queen5\_5.col &
25.0 & 0\% & 22.118 &
25.0 & 0\% & 0.005 &
25.0 & 0\% & 0.005 \\
queen6\_6.col &
36.0 & 0\% & 1356.976 &
36.0 & 0\% & 0.006 &
36.0 & 0\% & 0.006 \\
queen7\_7.col &
- & - & {\tt limit} &
49.0 & 0\% & 0.026 &
49.0 & 0\% & 0.026 \\
queen8\_8.col &
- & - & {\tt limit} &
64.0 & 0\% & 0.078 &
64.0 & 0\% & 0.078 \\
\hline
 \end{tabular}
 \caption{Results obtained by the three proposed formulations on the small instances of benchmark set {\bf dimacs}. 
 }
 \label{tab:dim:opt} 
\end{center}
\end{scriptsize}
\end{table}

\subsection{Comparison of the Greedy Heuristics}

The last round of computational experiments was designed to assess the overall quality of the greedy heuristics proposed in Section~\ref{sec::programming}. We execute both greedy strategies on both sets of instances ({\bf random} and {\bf dimacs}). The results of these tests are presented in Tables~\ref{tab:random:heur} and~\ref{tab:dim:heur}. In these tables, Column {\tt MIP} presents the results obtained by the asymmetric formulation which serves as a control for the quality of the solution obtained by the heuristics. Columns {\tt GREEDY\_MIN} and {\tt GREEDY\_MAX} present the results obtained by Algorithm~\ref{alg:heuristicmin} and Algorithm~\ref{alg:heuristicmax}, respectively. Each one of these larger columns is subdivided into two columns: {\tt SOL} - the number of colors used by the method, and {\tt TIME} - the elapsed time for each method, as measured by the solver for the MILP formulation and by the operational system for the greedy heuristics. 

\begin{table}[!h]
\begin{center}
\begin{scriptsize}
\begin{tabular}{rr|rr|rr|rr}
 & & \multicolumn{2}{|c|}{ MIP } & \multicolumn{2}{|c|}{ GREEDY\_MIN } & \multicolumn{2}{|c}{ GREEDY\_MAX } \\ 
\hline $|V(G)|$ & $p$ & BSOL & TIME & BSOL & TIME & BSOL & TIME \\ 
\hline
 10 & 0.05 & 3.0 & 0.001 & 3.6 & 0.358 & {\bf 3.0} & 0.359 \\ 
 10 & 0.10 & 4.8 & 0.001 & 5.8 & 0.359 & {\bf 5.0} & 0.357 \\ 
 10 & 0.20 & 5.2 & 0.002 & 6.2 & 0.365 & {\bf 5.8} & 0.365 \\ 
 10 & 0.30 & 7.2 & 0.001 & 8.0 & 0.380 & {\bf 7.6} & 0.388 \\ 
 10 & 0.40 & 7.8 & 0.001 & {\bf 8.2} & 0.384 & {\bf 8.2} & 0.381 \\ 
 20 & 0.05 & 5.6 & 0.026 & 8.2 & 0.394 & {\bf 6.0} & 0.391 \\ 
 20 & 0.10 & 7.4 & 0.035 & 10.8 & 0.385 & {\bf 9.2} & 0.385 \\ 
 20 & 0.20 & 10.8 & 0.032 & {\bf 12.6} & 0.361 & 12.4 & 0.362 \\ 
 20 & 0.30 & 14.4 & 0.016 & 16.0 & 0.368 & {\bf 15.8} & 0.369 \\ 
 20 & 0.40 & 19.0 & 0.003 & {\bf 19.0} & 0.357 & {\bf 19.0} & 0.358 \\ 
 30 & 0.05 & 9.0 & 16.436 & 15.0 & 0.376 & {\bf 11.2} & 0.370 \\ 
 30 & 0.10 & 11.0 & 1.782 & 17.6 & 0.390 & {\bf 14.4} & 0.386 \\ 
 30 & 0.20 & 16.4 & 0.430 & {\bf 19.6} & 0.373 & 20.2 & 0.369 \\ 
 30 & 0.30 & 23.4 & 0.037 & {\bf 24.8} & 0.364 & 25.4 & 0.361 \\ 
 30 & 0.40 & 29.2 & 0.006 & {\bf 29.4} & 0.325 & {\bf 29.4} & 0.326 \\ 
 40 & 0.05 & 10.6 & {\tt limit} & 22.2 & 0.388 & {\bf 14.0} & 0.376 \\ 
 40 & 0.10 & 14.6 & 1188.028 & 21.6 & 0.382 & {\bf 19.4} & 0.364 \\ 
 40 & 0.20 & 22.8 & 1.828 & {\bf 25.2} & 0.382 & 27.8 & 0.363 \\ 
 40 & 0.30 & 34.6 & 0.048 & {\bf 35.8} & 0.368 & 36.0 & 0.363 \\ 
 40 & 0.40 & 39.6 & 0.012 & {\bf 39.6} & 0.286 & {\bf 39.6} & 0.285 \\ 
 50 & 0.05 & 14.0 & {\tt limit} & 26.2 & 0.411 & {\bf 19.0} & 0.366 \\ 
 50 & 0.10 & 19.5 & 391.382 & 27.6 & 0.416 & {\bf 26.4} & 0.370 \\ 
 50 & 0.20 & 29.6 & 3.412 & {\bf 34.8} & 0.403 & 35.4 & 0.368 \\ 
 50 & 0.30 & 45.2 & 0.065 & {\bf 45.8} & 0.372 & 46.2 & 0.367 \\ 
 50 & 0.40 & 49.6 & 0.035 & {\bf 49.6} & 0.311 & {\bf 49.6} & 0.312 \\ 
 60 & 0.05 & 17.0 & 700.113 & 32.6 & 0.449 & {\bf 22.2} & 0.373 \\ 
 60 & 0.10 & 25.0 & 136.417 & {\bf 31.2} & 0.480 & 32.8 & 0.372 \\ 
 60 & 0.20 & 40.6 & 3.076 & {\bf 45.2} & 0.425 & {\bf 45.2} & 0.372 \\ 
 60 & 0.30 & 56.4 & 0.103 & {\bf 56.4} & 0.395 & 56.6 & 0.387 \\ 
 60 & 0.40 & 60.0 & 0.084 & {\bf 60.0} & 0.288 & {\bf 60.0} & 0.289 \\ 
 70 & 0.05 &  -  & {\tt limit} & 37.6 & 0.518 & {\bf 26.8} & 0.379 \\ 
 70 & 0.10 &  -  & {\tt limit} & {\bf 38.0} & 0.552 & 39.2 & 0.386 \\ 
 70 & 0.20 &  -  & {\tt limit} & {\bf 52.6} & 0.473 & 54.2 & 0.377 \\ 
 70 & 0.30 &  -  & {\tt limit} & {\bf 68.6} & 0.364 & {\bf 68.6} & 0.363 \\ 
 70 & 0.40 &  -  & {\tt limit} & {\bf 70.0} & 0.300 & {\bf 70.0} & 0.300 \\ 
 80 & 0.05 &  -  & {\tt limit} & 42.2 & 0.615 & {\bf 32.6} & 0.385 \\ 
 80 & 0.10 &  -  & {\tt limit} & {\bf 41.0} & 0.690 & 45.0 & 0.388 \\ 
 80 & 0.20 &  -  & {\tt limit} & {\bf 62.8} & 0.511 & 64.0 & 0.383 \\ 
 80 & 0.30 &  -  & {\tt limit} & {\bf 78.2} & 0.392 & {\bf 78.2} & 0.386 \\ 
 80 & 0.40 &  -  & {\tt limit} & {\bf 80.0} & 0.291 & {\bf 80.0} & 0.289 \\ 
 90 & 0.05 &  -  & {\tt limit} & 46.2 & 0.796 & {\bf 36.8} & 0.411 \\ 
 90 & 0.10 &  -  & {\tt limit} & {\bf 47.2} & 0.856 & 52.8 & 0.405 \\ 
 90 & 0.20 &  -  & {\tt limit} & {\bf 74.6} & 0.532 & 75.6 & 0.400 \\ 
 90 & 0.30 &  -  & {\tt limit} & {\bf 89.8} & 0.312 & {\bf 89.8} & 0.311 \\ 
 90 & 0.40 &  -  & {\tt limit} & {\bf 90.0} & 0.296 & {\bf 90.0} & 0.298 \\ 
 100 & 0.05 &  -  & {\tt limit} & 51.4 & 1.004 & {\bf 41.8} & 0.437 \\ 
 100 & 0.10 &  -  & {\tt limit} & {\bf 54.6} & 1.053 & 61.8 & 0.411 \\ 
 100 & 0.20 &  -  & {\tt limit} & {\bf 85.8} & 0.581 & 86.4 & 0.415 \\ 
 100 & 0.30 &  -  & {\tt limit} & {\bf 99.4} & 0.339 & {\bf 99.4} & 0.336 \\ 
 100 & 0.40 &  -  & {\tt limit} & {\bf 100.0} & 0.306 & {\bf 100.0} & 0.307 \\ 
 \hline
 \end{tabular}
 \caption{Results obtained by the two proposed greedy heuristics on the benchmark set {\bf random}. 
 }
 \label{tab:random:heur}
 \end{scriptsize}
\end{center}
\end{table}

Figure~\ref{fig:greedyres} summarizes the results from Table~\ref{tab:random:heur} (instances in the set {\bf random}). As done previously, the results were grouped by the probability of assigning an edge between a pair of vertices. In the vertical axis, we show the ratio between the best solution values for {\tt GREEDY\_MIN} and {\tt GREEDY\_MAX}. Regarding the instances with density between $30\%$ and $10\%$, we can argue that Algorithm~\ref{alg:heuristicmin} (represented by {\tt GREEDY\_MIN}) outperforms Algorithm~\ref{alg:heuristicmax} (represented by {\tt GREEDY\_MAX}) consistently, but curiously the opposite occurs for the instances with density $5\%$. This behavior seemed surprising at first but might be explainable by the fact that, with such low probability, the difference between the smallest and largest degree is very narrow. Indeed several vertices have the same degree. Then the choice of vertices is almost the same for both heuristics, which explains the variable behavior on the instances with density $5\%$.

\begin{figure}[!h]
    \centering
    \includegraphics[width=0.8\linewidth]{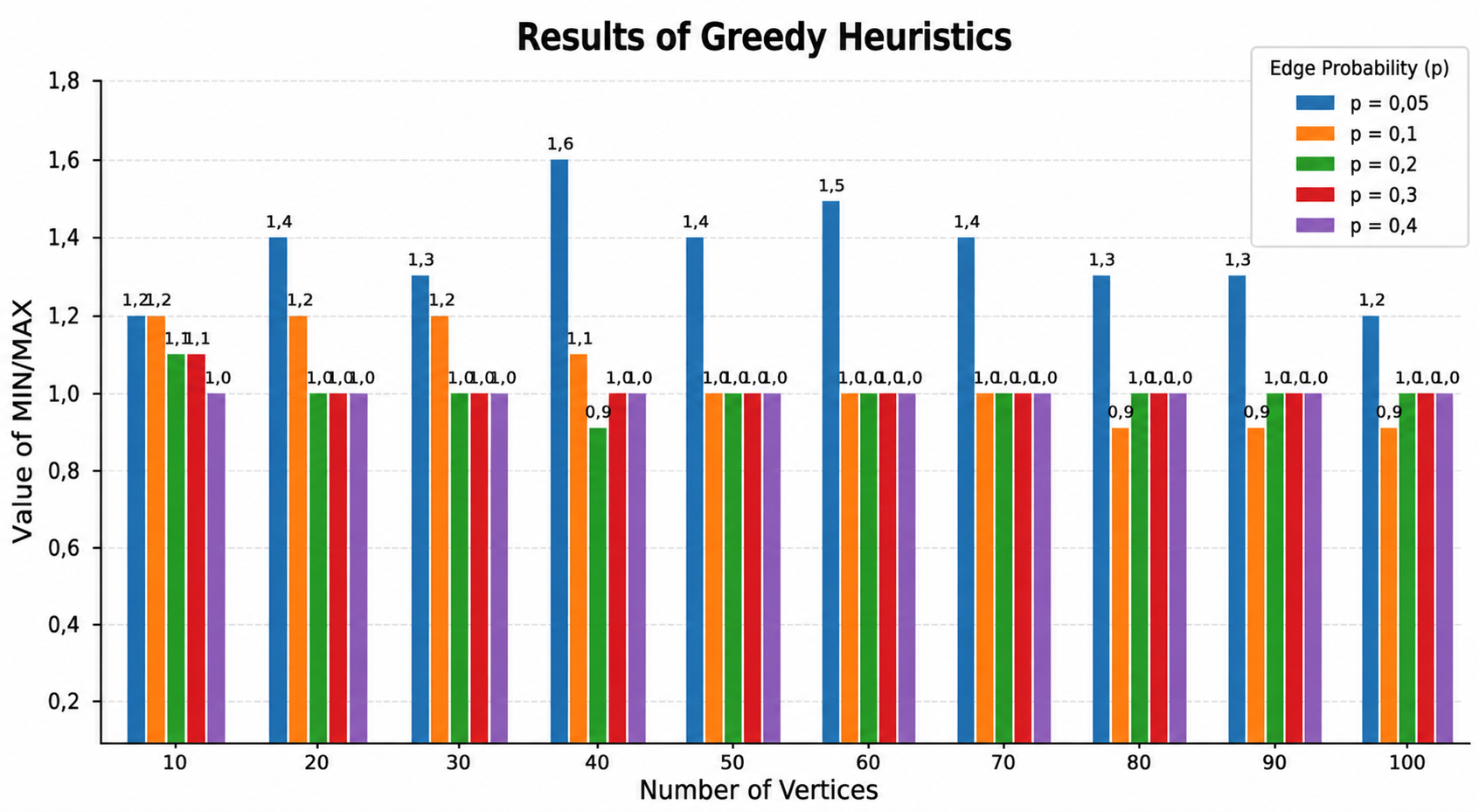}
    \caption{Sumary of the computational results for the set of instances {\tt random}. The vertical axis represents the ratio of the solution obtained by the $MIN$ greedy heuristic by the solution value obtained by the $MAX$ greedy heuristic.}
    \label{fig:greedyres}
\end{figure}

Table~\ref{tab:dim:heur} presents the results for the instances in the set {\bf dimacs}. First of all, notice that the heuristics are able to achieve the optimal solutions found by the formulations, except for {\em jean} and {\em huck}. Secondly, observe that {\tt GREEDY\_MAX} outperforms {\tt GREEDY\_MIN} in all instances except for {\em huck}. This difference is more pronounced in the class of instances {\em miles.i.X}.

\begin{table}[!h]
\begin{scriptsize}
\begin{center}
\begin{tabular}{l|rr|rr|rr}
 & \multicolumn{2}{|c|}{ MIP } & \multicolumn{2}{|c|}{ GREEDY\_MIN } & \multicolumn{2}{|c}{ GREEDY\_MAX } \\ 
\hline Instance & BSOL & TIME & BSOL & TIME & BSOL & TIME \\ 
\hline
anna.col &
- & {\tt limit} &
76.0  & 1.624 &
{\bf 72.0} & 0.462\\ 
david.col &
83.0 & 1.714 &
{\bf 83.0}  & 0.374 &
{\bf 83.0}  & 0.367\\
games120.col &
- & {\tt limit} &
72.0  & 1.360 &
{\bf 62.0} & 0.428 \\
huck.col &
 54.0 & 9.426 &
57.0  & 0.425 &
{\bf 54.0} & 0.367 \\
jean.col &
37.0 & 21.455 &
{\bf 39.0}  & 0.716 &
40.0 & 0.416 \\
miles1000.col &
- & {\tt limit} &
119.0  & 0.789 &
{\bf 118.0} & 0.465 \\
miles1500.col &
- & {\tt limit} &
{\bf 126.0}  & 0.534 &
{\bf 126.0} & 0.471 \\
miles250.col &
- & {\tt limit} &
93.0  & 0.849 &
{\bf 45.0} & 0.479 \\
miles500.col &
- & {\tt limit} &
99.0  & 0.995 &
{\bf 83.0} & 0.482 \\
miles750.col &
- & {\tt limit} &
109.0  & 0.935 &
{\bf 104.0} & 0.497 \\
mulsol.i.1.col &
- & {\tt limit} &
{\bf 137.0}  & 2.440 &
{\bf 137.0} & 0.736 \\
mulsol.i.2.col &
- & {\tt limit} &
{\bf 172.0}  & 0.666 &
{\bf 172.0} & 0.622 \\
mulsol.i.3.col &
- & {\tt limit} &
{\bf 173.0}  & 0.575 &
{\bf 173.0} & 0.604 \\
mulsol.i.4.col &
- & {\tt limit} &
{\bf 174.0}  & 0.573 &
{\bf 174.0} & 0.608 \\
mulsol.i.5.col &
- & {\tt limit} &
{\bf 175.0}  & 0.584 &
{\bf 175.0} & 0.617 \\
myciel3.col &
 11.0 & 0.003 &
{\bf 11.0}  & 0.271 &
{\bf 11.0} & 0.271 \\
myciel4.col &
23.0 & 0.001 &
{\bf 23.0}  & 0.267 &
{\bf 23.0} & 0.264 \\
myciel5.col &
47.0 & 0.013 &
{\bf 47.0}  & 0.268 &
{\bf 47.0} & 0.268 \\
myciel6.col &
- & {\tt limit} &
{\bf 95.0}  & 0.289 &
{\bf 95.0} & 0.288 \\
myciel7.col &
- & {\tt limit} &
{\bf 191.0}  & 0.367 &
{\bf 191.0} & 0.368 \\
queen10\_10.col &
- & {\tt limit} &
{\bf 100.0}  & 0.292 &
{\bf 100.0} & 0.293 \\
queen11\_11.col &
- & {\tt limit} &
{\bf 121.0}  & 0.346 &
{\bf 121.0} & 0.345 \\
queen12\_12.col &
- & {\tt limit} &
{\bf 144.0}  & 0.336 &
{\bf 144.0} & 0.338 \\
queen13\_13.col &
- & {\tt limit} &
{\bf 169.0}  & 0.365 &
{\bf 169.0} & 0.365 \\
queen14\_14.col &
- & {\tt limit} &
{\bf 196.0}  & 0.407 &
{\bf 196.0} & 0.403 \\
queen5\_5.col &
25.0 & 0.005 &
{\bf 25.0}  & 0.264 &
{\bf 25.0} & 0.262 \\
queen6\_6.col &
36.0 & {\tt 0.006} &
{\bf 36.0}  & 0.262 &
{\bf 36.0} & 0.261 \\
queen7\_7.col &
49.0 & {\tt 0.026} &
{\bf 49.0}  & 0.274 &
{\bf 49.0} & 0.271 \\
queen8\_12.col &
- & {\tt limit} &
{\bf 96.0}  & 0.290 &
{\bf 96.0} & 0.290 \\
queen8\_8.col &
64.0 & 0.078 &
{\bf 64.0}  & 0.280 &
{\bf 64.0} & 0.276 \\
queen9\_9.col &
- & {\tt limit} &
{\bf 81.0}  & 0.281 &
{\bf 81.0} & 0.282 \\
\hline
 \end{tabular}
 \caption{Results obtained by the two proposed greedy heuristics on the benchmark set {\bf dimacs}. 
 }
 \label{tab:dim:heur}
\end{center}
\end{scriptsize}
\end{table}

\section{Concluding remarks}

The harmonious coloring problem tends to be harder in sparse graphs, so it is not surprising that the solution methods tend to reach the optimum faster in dense graphs, as observed in the computational experiments.
The higher efficiency of the representatives models is mainly due to two facts: the number of variables in these models decreases with the edge density of $G^2$, and they are able to break some/all symmetries (especially the asymmetric version).
For $d$-degenerate graphs, we obtained a strong upper bound on the harmonious chromatic number, but we believe it could be further improved if expressed in terms of parameters other than the vertex cover and the maximum degree. The parameterized complexity of this problem has not been well studied yet. So, this would be a good research avenue to explore.

\section*{Acknowledgments}

A preliminary version of this work was presented at the VII Encontro de Teoria da Computação 2022, a satellite event of the yearly Congress of the Brazilian Computer Science Society. In the corresponding short abstract, written in Portuguese, one will find Lemma~\ref{lemad2}, the standard model, the representatives formulation and the asymmetric representatives formulation.
We also presented Theorem~\ref{idlema}, but it was incorrect. The statement is fixed here. A major part of this research was done during the second author's masters, which was financed by CNPQ 131142/2023-2.

\bibliographystyle{acm}
\bibliography{harmonious}

\end{document}